\documentclass[12pt]{article}
\usepackage{graphicx}
\usepackage{amsmath,amsthm,amssymb,enumerate}%, esint}
\usepackage{euscript,mathrsfs}
\usepackage[left=2cm,right=2cm,top=3.5cm,bottom=3.5cm]{geometry}
\usepackage{color}

\usepackage{soul}

\catcode`\@=11 \@addtoreset{equation}{section}

\catcode`\@=12

\allowdisplaybreaks

\newtheorem{Theorem}{Theorem}[section]
\newtheorem{Proposition}[Theorem]{Proposition}
\newtheorem{Lemma}[Theorem]{Lemma}
\newtheorem{Corollary}[Theorem]{Corollary}

\theoremstyle{definition}
\newtheorem{Definition}[Theorem]{Definition}

\newtheorem{Remark}[Theorem]{Remark}

\newcommand{\bTheorem}[1]{
\begin{Theorem} \label{T#1} }
\newcommand{\eT}{\end{Theorem}}

\newcommand{\bProposition}[1]{
\begin{Proposition} \label{P#1}}
\newcommand{\eP}{\end{Proposition}}

\newcommand{\bLemma}[1]{
\begin{Lemma} \label{L#1} }
\newcommand{\eL}{\end{Lemma}}

\newcommand{\bCorollary}[1]{
\begin{Corollary} \label{C#1} }
\newcommand{\eC}{\end{Corollary}}

\newcommand{\bRemark}[1]{
\begin{Remark} \label{R#1} }
\newcommand{\eR}{\end{Remark}}

\newcommand{\dy}{{\rm d}y}
\newcommand{\bDefinition}[1]{
\begin{Definition} \label{D#1} }
\newcommand{\eD}{\end{Definition}}

\newcommand{\Td}{\mathbb{T}^d}
\newcommand{\intTd}[1]{ \int_{\Td} #1 \ \D x}

\newcommand{\tvm}{\tilde{\vc{m}}}
\newcommand{\bfphi}{\boldsymbol{\varphi}}

\newcommand{\bFormula}[1]{
\begin{equation} \label{#1}}
\newcommand{\eF}{\end{equation}}

\newcommand{\Ov}[1]{\overline{#1}}

\newcommand{\aleq}{\stackrel{<}{\sim}}

\newcommand{\toS}{\stackrel{(S)}{\to}}
\newcommand{\eqS}{\stackrel{(S)}{\approx}}

\newcommand{\vr}{\varrho}

\newcommand{\tvr}{\tilde \vr}

\newcommand{\vm}{\vc{m}}

\newcommand{\vc}[1]{{\bf #1}}

\newcommand{\Div}{{\rm div}_x}
\newcommand{\Grad}{\nabla_x}

\newcommand{\dx}{\,{\rm d} {x}}

\newcommand{\dt}{\,{\rm d} t }

\newcommand{\dxdt}{\dx  \dt}

\newcommand{\D}{{\rm d}}

\newcommand{\ep}{\varepsilon}

\def\softd{{\leavevmode\setbox1=\hbox{d}%
          \hbox to 1.05\wd1{d\kern-0.4ex{\char039}\hss}}}
\definecolor{Cgrey}{rgb}{0.85,0.85,0.85}
\definecolor{Cblue}{rgb}{0.50,0.85,0.85}
\definecolor{Cred}{rgb}{1,0,0}
\definecolor{fancy}{rgb}{0.10,0.85,0.10}

\newcommand\Cbox[2]{%
    \newbox\contentbox%
    \newbox\bkgdbox%
    \setbox\contentbox\hbox to \hsize{%
        \vtop{
            \kern\columnsep
            \hbox to \hsize{%
                \kern\columnsep%
                \advance\hsize by -2\columnsep%
                \setlength{\textwidth}{\hsize}%
                \vbox{
                    \parskip=\baselineskip
                    \parindent=0bp
                    #2
                }%
                \kern\columnsep%
            }%
            \kern\columnsep%
        }%
    }%
    \setbox\bkgdbox\vbox{
        \color{#1}
        \hrule width  \wd\contentbox %
               height \ht\contentbox %
               depth  \dp\contentbox
        \color{black}
    }%
    \wd\bkgdbox=0bp%
    \vbox{\hbox to \hsize{\box\bkgdbox\box\contentbox}}%
    \vskip\baselineskip%
}

\date{}

\begin{document}

%%%%%%%%%%%%%%%%%%%%%%%%%%%%%%%%

\title{(S)--convergence and approximation of oscillatory solutions in fluid dynamics}

\author{Eduard Feireisl
\thanks{The work of E.F. was partially supported by the
Czech Sciences Foundation (GA\v CR), Grant Agreement
18--05974S. The Institute of Mathematics of the Academy of Sciences of
the Czech Republic is supported by RVO:67985840.} 
}

%\date{\today}

\maketitle

\centerline{Institute of Mathematics of the Academy of Sciences of the Czech Republic;}
\centerline{\v Zitn\' a 25, CZ-115 67 Praha 1, Czech Republic}

\centerline{Institute of Mathematics, Technische Universit\"{a}t Berlin,}
\centerline{Stra{\ss}e des 17. Juni 136, 10623 Berlin, Germany}
\centerline{feireisl@math.cas.cz}

\begin{abstract}

We propose a new concept of (S)--convergence applicable to 
numerical methods as well as other consistent approximations of  
the Euler system in gas dynamics. 
(S)--convergence, based on averaging in the spirit of Strong Law of Large Numbers, reflects the asymptotic properties of a given approximate sequence better than the standard description via Young measures. Similarity with the tools of ergodic theory is discussed.

\end{abstract}


{\bf MSC:} 
\bigskip

%\tableofcontents
\section{Introduction}
\label{i}

As illustrated by numerous recent results, the Euler system describing the motion of an inviscid fluid in the framework of continuum mechanics is (mathematically) ill--posed. This rather pessimistic conclusion applies to models of both incompressible and compressible fluids, see Bressan, Murray \cite{BresMur}, Buckmaster, Vicol \cite{BucVic}, Chiodaroli \cite{Chiod}, DeLellis, Sz\' ekelyhidi \cite{DelSze3}, Wiedemann \cite{Wied}, among many others. 
Accordingly, solutions of the Euler system should be perceived as limits of physically grounded and mathematically well posed approximations, among which various types of \emph{zero dissipation limits}. A prominent example are approximations by \emph{numerical schemes}, where the numerical viscosity mimics the physical viscosity present in \emph{real} fluids.  

To study the properties of approximate solutions, we introduce the concept of \emph{consistent approximation} borrowed from numerics, where the underlying equations are satisfied modulo small consistency errors vanishing in the asymptotic limit. We focus on the case where the approximate      
solutions are only bounded, meaning they may exhibit oscillations (called wiggles in numerics) and/or concentrations. In view of the results shown in \cite{BreFeiHof19}, \cite{MarEd}, oscillations and/or concentration appear if, simultaneously, 
 \begin{itemize}
\item the limit Euler system does not admit a regular solution -- a consequence of the general 
weak--strong uniqueness principle stated e.g. in Gwiazda et al. \cite{GSWW}; 
\item all accumulation points of the approximate sequence are not smooth, see \cite{BreFeiHof19}; 
\item accumulation points of the approximate sequence are not (weak) solutions of the limit system as otherwise 
the convergence would be strong, see \cite{MarEd}. 
\end{itemize}
Anticipating the above scenario we recall the work of DiPerna and Majda \cite{DiPMaj87a}, \cite{DiP2} on measure--valued solutions, where
suitable solutions of the Euler system are identified with the Young measure generated by an approximating sequence. More recently, there has been a series of attempts to compute and visualize the measure--valued solutions to the Euler system, see Fjordholm et al. 
\cite{FjKaMiTa}, \cite{FjMiTa1} or \cite{FeiLuk}. 

The main aim of the present paper is to propose a different strategy based on statistical averaging, inspired by the concept 
of $\mathcal{K}$--convergence, cf. Balder \cite{Bald}, Koml\'{o}s \cite{Kom}. To illuminate the main idea, consider a simple sequence of functions:
\[
U_n = 1 \ \mbox{if}\ n \ \mbox{is odd},\ U_n = 0 \ \mbox{if}\ n \ \mbox{is even.}
\]
Of course, more sophisticated and physically relevant examples can be produced. Consider the asymptotic limit 
of $\{ U_n \}_{n=1}^\infty$ for $n \to \infty$. A naive but natural solution would be that the ``limit'' is a convex combination of 
Dirac masses, specifically, 
\begin{equation} \label{i1}
U_n \to \frac{1}{2} \delta_{1} + \frac{1}{2} \delta_{0}, 
\end{equation}
where $\delta_X$ denotes the Dirac mass at $X$. However, the standard approach based on the theory of Young measures provides a different 
conclusion: $U_{n_k} \to 1$ for a suitable subsequence ($n_k$ odd), or $U_{n_l} \to 0$ for another suitable subsequence
($n_l$ even). Choosing 
a ``suitable'' subsequence we get either 
\[
U_{n_k} \to 1 \ \mbox{as}\ k \to \infty, 
\]
or 
\[
U_{n_l} \to 0  \ \mbox{as}\ l \to \infty.
\]
Intuitively, the limit \eqref{i1} reflects better the asymptotic properties of $\{ U_n \}_{n=1}^\infty$, however, it is never seen if the standard Young measure based theory is applied. The goal of this paper is to introduce the concept of statistical limit that will identify \eqref{i1} as the only eligible option.

We claim that convergence ``up to a subsequence'', that may be useful in theoretical studies, is of no practical use in numerical experiments, where unconditional convergence is implicitly assumed. What is more, the weak convergence, meaning convergence in 
the sense of integral averages, is difficult to visualize as well. The same applies to Young measures that are necessarily objects 
resulting from a weak limit.

Our goal is to introduce a concept of statistical (S)--convergence of approximate solutions, where the limit is identified with 
a parametrized measure representing a generalized solution of the target problem. The convergence is strong (a.a) with respect to the physical space and in the Wasserstein metric on the space of probability measures. In particular, all observable quantities like the barycenter, deviation, variance as well as higher order moments of the limit measure can be identified as limits of \emph{strongly} converging sequences. Moreover, 
(S)--convergence is robust with respect to \emph{statistical} perturbations of the approximate sequence.

A sufficient but definitely not necessary condition for a sequence to be (S)--convergent is its \emph{asymptotic stationarity}. This means, very roughly indeed, that statistical distribution of observable quantities features certain ergodicity. In particular, for 
\emph{strongly} (S)--convergent sequences of measurable functions $\{ \vc{U}_n \}_{n=1}^\infty$, the limit of ergodic averages 
\[
\frac{1}{w_N} \sum_{n=1}^N w \left( \frac{n}{N} \right) b (\vc{U}_n) \to \Ov{b(\vc{U})} \ \mbox{in}\ L^1 
\ \mbox{as}\ N \to \infty,\ w_N \equiv \sum_{n=1}^N w \left( \frac{n}{N} \right),  
\]
exists for any bounded continuous function $b$ and any weight function $w \in C^1[0,1]$, $w \geq 0$, $\int_0^1 w(z)\ \D z = 1$. 
The mapping $\mathcal{V}: b \mapsto \Ov{b(\vc{U})}$ can be identified with a (parametrized) measure -- the (S)--limit of $\{ \vc{U}_n \}_{n=1}^\infty$. 

Our working plan is as follows:

\begin{itemize}

\item In Section \ref{S}, we introduce the concept of {\bf (S)--convergence} and discuss its basic properties. 

\item In Section \ref{Q} we show that (S)--convergence is {\bf stable with respect to statistical perturbations}. 

\item

Furthermore, 
we discuss several examples of (S)--converging sequences, in particular perturbations of {\bf stationary approximations} in 
Section \ref{A}. We also establish several sufficient conditions for a sequence to be (S)--convergent in terms of asymptotic 
stationarity.

\item Finally, in Section \ref{E}, we apply the abstract theory to approximate sequences of the {\bf isentropic Euler system}.

\end{itemize}

\section{Statistical limit, (S)--convergence}
\label{S}

Let $Q \subset R^M$ denote the \emph{physical space}. In the context of fluid mechanics, we usually consider 
\[
Q = \left\{ (t,x) \ \Big| \ t \in (0,T) ,\ x \in \Omega \right\},   
\]
where $t$ denotes the time, and $x \in \Omega \subset R^d$, $d=1,2,3$, is the space coordinate confined to a physical domain $\Omega$ 
occupied by the fluid. The state of the system will be 
denoted by $\vc{U} = \vc{U}(t,x)$. For the Euler system, the vector $\vc{U}$ contains the basic state variables, for instance 
$\vc{U} = [\vr, \vm, S]$, where $\vr$ is the mass density, $\vm$ is the momentum, and $S$ is the total entropy. Accordingly, 
$\{ \vc{U}_n \}_{n=1}^\infty$ denotes a sequence of approximate solutions.

We introduce the basic concept of (S)--convergence inspired by the theory of stationary stochastic processes, see e.g. Krylov 
\cite[Chapter 4]{Krylovstoch}, and the approach developed by Das and Yorke \cite{DaYo}.  

\begin{Definition} [(S)--convergence] \label{SD1}

Let $\vc{U}_n: Q \to R^D$, $n=1,2,\dots$ be a sequence of measurable functions. 

\noindent
{\bf (i)} 
We say that $\{ \vc{U}_n \}_{n=1}^\infty$ 
is \emph{weakly (S)--convergent} if for any $b \in C_c(R^D)$ there holds: 
\begin{itemize}
\item {\bf Weak correlation limit} 
\begin{equation} \label{S1}
\lim_{N \to \infty} \frac{1}{N} \sum_{n=1}^N \int_Q b(\vc{U}_n) b(\vc{U}_m) \ \D y \ \mbox{exists}
\end{equation}
for any fixed $m$; 
\item {\bf Weak correlation disintegration} 
\begin{equation} \label{S2}
\lim_{N \to \infty} \sum_{n,m=1}^N \frac{1}{N^2} \int_Q b(\vc{U}_n) b(\vc{U}_m) \ \D y =
\lim_{M \to \infty} \frac{1}{M} \sum_{m=1}^M \left[ \lim_{N \to \infty} \frac{1}{N} \sum_{n=1}^N \int_Q b(\vc{U}_n) 
b(\vc{U}_m) \ \D y \right] .
\end{equation}

\end{itemize}

\medskip 

\noindent
{\bf (ii)} 
We say that $\{ \vc{U}_n \}_{n=1}^\infty$ is \emph{strongly (S)--convergent} if for any $b \in C_c(R^D)$ and any 
\[
w \in 
W = \left\{ w \in C^1[0,1] \ \Big| \ w \geq 0, \ \int_0^1 w(z) \ \D z = 1 \right\}
\]
there holds:
\begin{itemize}
\item {\bf Strong correlation limit} 
\begin{equation} \label{Su1}
\lim_{N \to \infty} \frac{1}{w_N} \sum_{n=1}^N \int_Q w \left( \frac{n}{N} \right) b(\vc{U}_n) b(\vc{U}_m) \ \D y,\ 
w_N \equiv \sum_{n=1}^N w \left( \frac{n}{N} \right),
 \ \mbox{exists}
\end{equation}
for any fixed $m$ and is independent of $w$;
\item {\bf Strong correlation disintegration} 
\begin{equation} \label{Su2}
\begin{split}
\lim_{N \to \infty} &\sum_{n,m=1}^N \frac{1}{w_N^2} w \left( \frac{n}{N} \right) 
w \left( \frac{m}{N} \right) \int_Q b(\vc{U}_n) b(\vc{U}_m) \ \D y \\ &=
\lim_{M \to \infty} \frac{1}{w_M} \sum_{m=1}^M w \left( \frac{m}{M} \right)\left[ \lim_{N \to \infty} \frac{1}{w_N} \sum_{n=1}^N w \left( \frac{n}{N} \right) \int_Q b(\vc{U}_n) 
b(\vc{U}_m) \ \D y \right] 
\end{split}
\end{equation}
for any $w \in W$.

\end{itemize}

\end{Definition}

Obviously \emph{ strong $\Rightarrow$ weak } as the choice $w \equiv 1$ yields the desired conclusion. 
Conditions \eqref{S1}, \eqref{S2} are rather mild, and, as we shall see below, allow for a large class approximate sequences that comply 
with a generalized form of Strong Law of Large Numbers: 
\[
\lim_{N \to \infty} \frac{1}{N} \sum_{n=1}^N b(\vc{U}_n) \ \mbox{exists for any}\ b \in C_c(R^D),
\]
see Theorem \ref{SL2} below. Note that \eqref{S1} and even \eqref{Su1} hold as soon as the sequence $\{ \vc{U}_n \}_{n=1}^\infty$ generates a Young measure, 
specifically, 
\[
b(\vc{U}_n) \to \Ov{b(\vc{U})} \ \mbox{as}\ n \to \infty \ \mbox{for any}\ b \in C_c(R^D), 
\]
The opposite implication in general fails. Indeed it is easy to construct a sequence of functions (real numbers) that 
give rise to (infinitely many) different Young measures generated by different subsequences for which \eqref{S1}, \eqref{S2} 
or even \eqref{Su1}, \eqref{Su2} still hold.

\subsection{Limit of an (S)--convergent sequence}
  
Our next goal is to associate a limit to an (S)--convergent sequence -- a parametrized measure 
\[
\mathcal{V} \in L^\infty_{{\rm weak-(*)}}(Q; \mathcal{M}^+(R^D))
\]
that can be seen as an analogue of the Young measure for weakly converging sequences. To this end, a few preliminary observations are 
needed.

\begin{Lemma} \label{SL1}

Let $b \in C_c(R^D)$ and $w \in W$ be given. 

The following is equivalent:

\begin{itemize}
\item the correlation limit \eqref{Su1} exists for any $m$; 
\item 
\begin{equation} \label{S3}
\frac{1}{w_N} \sum_{n=1}^N w \left( \frac{n}{N} \right)  b(\vc{U}_n) \to \Ov{b(\vc{U})} \ \mbox{weakly-(*) in}\ 
L^\infty(Q) \ \mbox{as}\ N \to \infty.
\end{equation}

\end{itemize}

Moreover, if the correlation limit is independent of $w$, then so is the weak limit $\Ov{b(\vc{U})}$.

\end{Lemma}

\begin{proof}

Obviously \eqref{S3} $\Rightarrow$ \eqref{Su1} for any $m$; whence it is enough to show the opposite implication.  
Write the Hilbert space $L^2(Q)$ as 
\[
L^2(Q) = Y \oplus Y^\perp, 
\]
where 
\[
Y \equiv \Ov{ {\rm span} \left\{ b(\vc{U}_m)\ \Big|\ m=1,2,\dots \right\} }^{L^2(Q)}
\]
Denote 
\[
B_N \equiv \frac{1}{w_N} \sum_{n=1}^N w \left( \frac{n}{N} \right) b(\vc{U}_n). 
\]
Obviously $B_N \in Y$, and, by virtue of \eqref{Su1}, 
\[
\lim_{N \to \infty} \int_Q B_N \ \phi \ \D y \ \mbox{exists for any} \ \phi \in Y.
\]
On the other hand, 
\[
\int_Q B_N \ v \ \D y = 0 \ \mbox{whenever}\ v \in Y^\perp.
\]
Writing any $u \in L^2(Q)$ as $u = \phi \oplus \phi^\perp$ we may infer that 
\[
B_N \to \Ov{b(\vc{U})} \ \mbox{weakly in}\ L^2(Q) \ \mbox{as}\ N \to \infty.
\]
As $b(\vc{U}_n)$ are uniformly bounded, this yields \eqref{S3}.

\end{proof}

The strong correlation limit postulated in \eqref{Su1} may seem difficult to check. It can be simplified as stated in the following 
assertion.

\begin{Lemma} \label{SCL}

Let $\{ \vc{U}_n \}_{n=1}^\infty$ be a sequence of measurable functions, $\vc{U}_n: Q \to R^D$. Let $b \in C_c(R^D)$ be given. 

Then the following is equivalent:

\begin{itemize}

\item the strong correlation limit \eqref{Su1} exists and is independent of $w \in W$; 

\item the limit

\begin{equation} \label{SCL1}
\frac{1}{\beta - \alpha} \lim_{N \to \infty} \frac{1}{N} \sum_{\alpha N \leq n \leq \beta N} \int_Q b(\vc{U}_n) b(\vc{U}_m) \ \D y   
\end{equation}
exists for any fixed $m$, any $0 \leq \alpha < \beta \leq 1$, and is independent of $\alpha, \beta$.

\end{itemize}

\end{Lemma}  

\begin{proof}

Consider the function  
\[
B_N(z, \phi) = \int_Q b(\vc{U}_n) \phi \ \D y \ \mbox{for}\ z \in \left[ \frac{n-1}{N}, \frac{n}{N} \right),\ 
1 \leq n \leq N, \ z \in [0,1) 
\]
In accordance with Lemma \ref{SL1}, the existence of the strong correlation limit \eqref{S3} can be equivalently stated as 
\[
\int_0^1 w(z) B_N(z, \phi) \ \D z \to \int_0^1 w(z) \int_Q \Ov{ b (\vc{U} )} \phi \ \D y 
\ \D z  \ \mbox{for any}\ 
\phi \in L^1(Q), \ w \in W.
\]
In other words 
\begin{equation} \label{sscs}
B_N(z, \phi)  \to \int_Q \Ov{ b (\vc{U} )} \phi \ \D y \ \mbox{weakly-(*) in}\ L^\infty(0,1), 
\end{equation}
where the hypothesis that the correlation limit is independent of $w$ is reflected by the fact that the weak limit 
$\int_Q \Ov{ b (\vc{U} )} \phi \ \D y$ is independent of $z$. In fact, relation \eqref{sscs} is equivalent with 
\eqref{Su1} which yields the desired conclusion.

\end{proof}

We are ready to define a limit for an (S)--convergent sequence.
Identifying 
\[
L^\infty_{{\rm weak - (*)}} (Q; \mathcal{M}(R^D)) 
\ \mbox{with the dual} \ \left[ L^1(Q; C_0 (R^D)) \right]^*
\]
we define 
\[
\left< \mathcal{V}; \varphi \ b \right> = \int_Q \varphi \ \Ov{b(\vc{U})} \ \D y \ \mbox{for}\ \varphi \in L^1(Q),\ 
b \in C_c(R^D),  
\]
where $\Ov{b(\vc{U})}$ is the (unique) limit identified in \eqref{S3}. This defines a unique linear functional on 
$L^1(Q; C_0 (R^D))$ via standard density argument. Obviously, the limit measure is bounded non--negative, 
\[
\mathcal{V} \in L^\infty_{{\rm weak-(*)}}(Q; \mathcal{M}^+(R^D)).
\] 
The measure $\mathcal{V}$ will be termed \emph{(S)--limit} of the sequence $\{ \vc{U}_n \}_{n=1}^\infty$, and we shall write 
\begin{equation} \label{S3a}
\vc{U}_n  \toS \mathcal{V}.
\end{equation}

\subsection{Equivalence with convergence of ergodic means}

To give a more specific meaning to \eqref{S3a}, we need the following result.

\begin{Theorem}[Equivalence principle] \label{SL2}
\

{\bf (i)} The following is equivalent:

\begin{itemize}

\item $\{ \vc{U}_n \}_{n=1}^\infty$ is weakly (S)--convergent;
\item 
\begin{equation} \label{S4}
\frac{1}{N} \sum_{n=1}^N b(\vc{U}_n) \to \Ov{b (\vc{U})} \ \mbox{(strongly) in}\ L^1(Q) 
\end{equation}
for any $b \in C_c(R^D)$.

\end{itemize}

\medskip 

{\bf (ii)} The following is equivalent:

\begin{itemize}

\item $\{ \vc{U}_n \}_{n=1}^\infty$ is strongly (S)--convergent;
\item 
\begin{equation} \label{Su4}
\frac{1}{w_N} \sum_{n=1}^N w \left( \frac{n}{N} \right) b(\vc{U}_n) \to \Ov{b (\vc{U})} \ \mbox{(strongly) in}\ L^1(Q) 
\end{equation}
for any $b \in C_c(R^D)$, and any $w \in W$, where the limit is independent of $w$.

\end{itemize}

\end{Theorem}

\begin{proof}

It is enough to show the equivalence for strongly (S)--convergent sequences. 
As \eqref{Su4} clearly implies strong (S)--convergence, we focus on the opposite implication. 
By virtue of Lemma \ref{SL1}, we have 
\[
B_N \equiv \frac{1}{w_N} \sum_{n=1}^N w \left( \frac{n}{N} \right) b(\vc{U}_n) \to \Ov{b(\vc{U})} \ \mbox{weakly-(*) in}\ L^\infty(Q),
\]
where the limit is independent of the weight function $w$.
Accordingly, relation \eqref{Su4} can be reformulated as 
\[
\int_Q |B_N|^2 \ \D y \to \int_Q |\Ov{b(\vc{U})} |^2 \D y \ \mbox{as}\ N \to \infty,
\]
which is nothing other than \eqref{Su2}.

\end{proof}

\begin{Remark} \label{SR1}

As $b$ is bounded and $|Q| < \infty$, the convergence in \eqref{Su4} can be replaced by 
\[
\frac{1}{w_N} \sum_{n=1}^N w \left( \frac{n}{N} \right) b(\vc{U}_n) \to \Ov{b (\vc{U})} \ \mbox{in}\ L^q(Q)
\]
for any $1 \leq q < \infty$.

\end{Remark}

Going back, we may rewrite \eqref{S3a} in the form 
\begin{equation} \label{S4a}
\vc{U}_n \toS \mathcal{V} \ \Leftrightarrow \ \int_Q \left| d_{{\rm weak-(*)}} \left[ 
\frac{1}{N} \sum_{n=1}^N \delta_{\vc{U}_n(y)} ; \mathcal{V}_y \right] \right| \D y \to 0 \ \mbox{as}\ N \to \infty, 
\end{equation}
where $\delta$ is the Dirac mass, and 
$d_{{\rm weak}-(*)}$ denotes the metric on a bounded ball in $\mathcal{M}(R^D)$ endowed with the weak-(*) topology. 

Here, the relation $\toS$ is considered for weakly (S)--convergent sequences. If, in addition, the convergence is strong, we may 
identify the limit measure in \eqref{S4a} through  
\begin{equation} \label{S4ab}
\int_Q \left| d_{{\rm weak-(*)}} \left[ 
\frac{1}{w_n} \sum_{n=1}^N w \left( \frac{n}{N} \right) \delta_{\vc{U}_n(y)} ; \mathcal{V}_y \right] \right| \D y \to 0 \ \mbox{as}\ N \to \infty
\end{equation}
choosing \emph{arbitrary} weight function $w$. As observed by Krengel \cite{Kren}, the \emph{rate of convergence} for (S)--converging sequences can be rather slow, which may be embarassing in numerical applications. Das and Yorke \cite{DaYo} therefore 
proposed to introduce the weights $w$ and obtained super convergence for certain quasiperiodic sequences. The possibility of 
choosing suitable $w$ can possibly improve the rate of convergence of numerical approximations.

Although the analogy between $\mathcal{V}$ and the Young measure is obvious, we note that an $(S)-$convergent sequence 
may not generate any Young measure. Still the following result holds:
\begin{Proposition} \label{SP1}

Let $\{ \vc{U}_n \}_{n=1}^\infty$ be uniformly integrable, meaning 
\[
\int_Q F(|\vc{U}_n|) \ \D y \aleq 1 \ \mbox{uniformly for}\ n \to \infty,
\]
where $F: [0, \infty) \to [0, \infty)$ is continuous, $\lim_{r \to \infty} F(r) = \infty$.

Then there is a subsequence $\{ \vc{U}_{n_k} \}_{k=1}^\infty$ that is weakly (S)--convergent and generates a Young measure that coincides with 
its (S)--limit.

\end{Proposition}

\begin{proof} 

Using the celebrated Koml\' os theorem (see Koml\' os \cite{Kom}), 
Balder \cite{Bald} showed the existence of a subsequence satisfying \eqref{S4} for any $b \in C_c(R^D)$; whence the desired 
conclusion follows from Lemma \ref{SL2}.

\end{proof}

\subsection{Convergence in Wasserstein distance}

Up to now, we did not impose any further hypotheses concerning integrability of the generating sequence. In applications, 
integrability is usually assumed to ensure finiteness of the first moments, and, in particular, to identify the barycenter of
the measure  
$\mathcal{V}$. 

\begin{Proposition} \label{QP1}

Let $\{ \vc{U}_n \}_{n=1}^\infty$ be a sequence of measurable functions satisfying 
\[
\int_Q { |\vc{U}_n |^p \ \D y } \leq c \ \mbox{uniformly for}\ n=1,2,\dots, \ p \geq 1. 
\]
Let 
\[
\vc{U}_n \toS \mathcal{V}.
\]

Then 
\begin{itemize}
\item 
$\mathcal{V}_y$ is a probability measure on $R^D$ for a.a. $y \in Q$, with finite moments of order $p$.

\item If $p > 1$, then 
\[
\int_Q \left| d_{W_s} \left[ \frac{1}{N} \sum_{n=1}^N \delta_{\vc{U}_n(y)}; \mathcal{V}_y \right] \right|^s \ \D y
\to 0 \ \mbox{as}\ N \to \infty, 
\]
where $d_{W_s}$ denotes the Wasserstein distance of $s-$th order.

In particular, the barycenters converge,  
\[
\frac{1}{N} \sum_{n=1}^N \vc{U}_n \to \left< \mathcal{V}; \widetilde{\vc{U}} \right> \equiv \vc{U} 
\ \mbox{as}\ N \to \infty,
\]
in $L^s(Q)$, $1 \leq s < q$.

\end{itemize}

\end{Proposition}

\begin{proof} 

The proof is quite standard and can be done in the same way as for the Young measures, see e.g. 
\cite{FeiLukMizSheWa}.

\end{proof}

\begin{Remark} \label{RRS11}

If the sequence $\{ \vc{U}_n \}_{n=1}^\infty$ in Proposition \ref{QP1} is strongly (S)--convergent, then 
\[
\frac{1}{N} \sum_{n=1}^N \ \mbox{may be replaced by} \ \frac{1}{w_N} \sum_{n=1}^N w \left( \frac{n}{N} \right)
\]
for any weight $w \in W$.

\end{Remark}

As a matter of fact, the measure $\mathcal{V}$ is a probability measure under very mild assumption of uniform integrability 
stated in Proposition \ref{SP1}. If $s = 1$, the celebrated Koml\' os theorem (see \cite{Kom}) yields  a \emph{subsequence} $\{ \vc{U}_{n_k} \}_{k=1}^\infty$ such that 
\[
\frac{1}{N} \sum_{k=1}^N \vc{U}_{n_k} \to \vc{U}_K \ \mbox{a.a. in}\ Q. 
\] 
Moreover, the theory of Balder \cite{Bald} asserts that this subsequence can be chosen in such a way that 
\[
\vc{U}_{n_k} \toS \mathcal{V}_K \ \mbox{as}\ k \to \infty, 
\]
where $\mathcal{V}_K$ is the Young measure generated by $\{ \vc{U}_{n_k} \}_{k=1}^\infty$, cf. Proposition \ref{SP1}. Note, however, that there may be different subsequences generating different limits; whence, in general 
\[
\vc{U}_K \ne \vc{U},\ \mathcal{V}_K \neq \mathcal{V}.
\] 

To conclude, let us point out, that the strength of the concept of (S)--convergence lies in the fact there is no need for subsequence, which is 
particularly relevant in numerical applications. 
Moreover, statistical deviations are eliminated and a large class of approximate sequences can be accommodated as we shall show in the 
next section.

\section{Robustness with respect to statistical perturbations}
\label{Q} 

We start by introducing the concept of statistically equivalent sequences, see e.g. 
Le\' on--Saavedra et al. \cite{SaLiFe}.

\begin{Definition}[Statistical equivalence] \label{DQ1}

We say that two sequences $\{ \vc{U}_n \}_{n=1}^\infty$, $\{ \vc{V}_n \}_{n=1}^\infty$ of measurable 
functions are \emph{statistically equivalent}, 
\[
\{ \vc{U}_n \}_{n=1}^\infty \eqS \{ \vc{V}_n \}_{n=1}^\infty, 
\]
if for any $\ep > 0$, 
\begin{equation} \label{Q1}
\frac{ \# \left\{ k \leq N\ \Big| \int_Q |\vc{U}_n - \vc{V}_n | \ \D y > \ep \right\} }{N}
\to 0 \ \mbox{as}\ N \to \infty.
\end{equation}

\end{Definition}

Next we show that statistical perturbation do not influence the (S)--convergence. 

\begin{Theorem}[Statistical perturbation] \label{QT1}

Suppose that 
\[
\{ \vc{U}_n \}_{n=1} \eqS \{ \vc{V}_n \}_{n=1}^\infty, 
\]
and that 
\[
\vc{U}_n \toS \mathcal{V}.
\]

Then 
\[
\vc{V}_n \toS \mathcal{V}.
\]

\end{Theorem}

\begin{proof}

In view of Lemma \ref{SL2}, we have to show that 
\begin{equation} \label{Q2}
\frac{1}{N} \sum_{n=1}^N b(\vc{V}_n) = \Ov{b(\vc{V})} = \Ov{b(\vc{U})} = \frac{1}{N} \sum_{n=1}^N b(\vc{U}_n) 
\ \mbox{for any}\ b \in C_c(R^D).
\end{equation}
Suppose $b \in C_c(R^D)$ is Lipschitz.
Choose $\ep > 0$ and define the set  
\[
r(\ep) = \left\{ n \ \mbox{positive integer}\ \Big| \ \| \vc{V}_n - \vc{U}_n \|_{L^1(Q)} \leq \ep \right\}.
\]
Now, compute 
\[
\frac{1}{N} \sum_{n=1}^N b(\vc{V}_n ) = \frac{1}{N} \sum_{ n \leq N, n \in r(\ep) } b(\vc{V}_n ) 
+ \frac{1}{N} \sum_{ n \leq N, n \notin r(\ep) } b(\vc{V}_n ), 
\]
where, in view of \eqref{Q1}, 
\[
\frac{1}{N} \sum_{ n \leq N, n \notin r(\ep) } b(\vc{V}_n ) \to 0 \ \mbox{in}\ L^\infty(Q) \ \mbox{as}\ N \to \infty
\]
as $b$ is bounded.

Next, 
\[
\frac{1}{N} \sum_{ n \leq N, n \in r(\ep) } b(\vc{V}_n ) = 
\frac{1}{N} \sum_{ n \leq N, n \in r(\ep) } b(\vc{U}_n ) +  \frac{1}{N} \sum_{ n \leq N, n \in r(\ep) } \Big[ b(\vc{V}_n ) - b(\vc{U}_n)
\Big],
\]
where, by the same token as above, 
\[
\lim_{N \to \infty} \frac{1}{N} \sum_{ n \leq N, n \in r(\ep) } b(\vc{U}_n ) = 
\lim_{N \to \infty} \frac{1}{N} \sum_{ n =1}^N b(\vc{U}_n ) = \Ov{b(\vc{U})} \ \mbox{in}\ L^1(Q). 
\]

Finally, as $b$ is Lipschitz, 
\[
\left\| 
\frac{1}{N} \sum_{ n \leq N, n \in r(\ep) } \Big[ b(\vc{V}_n ) - b(\vc{U}_n)
\Big] \right\|_{L^1(Q)} \leq \frac{1}{N} \sum_{ n \leq N, n \in r(\ep) } \left\|  b(\vc{V}_n ) - b(\vc{U}_n)
\right\|_{L^1(Q)} \aleq \ep. 
\]
As $\ep > 0$ was arbitrary, this shows \eqref{Q2} for any Lipschitz $b$. The rest follows by standard density argument.

\end{proof}

\begin{Remark} \label{RR1}

Theorem \ref{QT1} refers to weakly (S)--convergent sequences. However, the same argument can be used to show that if 
\[
\{ \vc{U}_n \}_{n=1} \eqS \{ \vc{V}_n \}_{n=1}^\infty, 
\]
then 
\[
\{ \vc{U}_n \}_{n=1} \ \mbox{strongly (S)--convergent} \ \Leftrightarrow \ 
\{ \vc{V}_n \}_{n=1}^\infty \ \mbox{strongly (S)--convergent}.
\]

We immediately get that (S)--convergence accommodates strong convergence as the case may be. 

\end{Remark}

\begin{Corollary}[Strongly convergent perturbation] \label{QC1}
\  

\begin{itemize}

\item 
Suppose that 
\[
\vc{U}_n \to \vc{U} \ \mbox{in}\ L^1(Q). 
\]

Then 
\[
\{ \vc{U}_n \}_{n=1}^\infty \ \mbox{is strongly (S)--convergent, and}\ 
\vc{U}_n \toS \delta_{\vc{U}}.
\]

\item Suppose that $\{ \vc{U}_n \}_{n=1}^\infty$ is weakly/strongly (S)--convergent, and 
\[
\vc{V}_n \to 0 \ \mbox{in}\ L^1(Q). 
\]

Then $\{ \vc{U}_n + \vc{V}_n \}_{n=1}^\infty$ is weakly/strongly (S)--convergent, and 
(S)--converge to the same measure $\mathcal{V}$.

\end{itemize}

\end{Corollary}

\section{Asymptotically stationary approximations}
\label{A}

There is an important class of (S)--convergent sequences called stationary in stochastic terminology. 

\begin{Definition}[Stationary sequence] \label{AD1}

A sequence $\{ \vc{U}_n \}_{n=1}^\infty,\ \vc{U}_n : Q \mapsto R^D$, is called \emph{stationary} if 
\begin{equation} \label{A1}
\int_Q B(\vc{U}_{k_1}, \dots, \vc{U}_{k_j} ) \ \D y = 
\int_Q B (\vc{U}_{k_1+n }, \dots, \vc{U}_{k_j + n} ) \ \D y
\end{equation}
for any $B \in C_c(R^{jD})$, any $1 \leq k_1 \leq k_2 \dots \leq k_j$, $j \geq 1$, $n \geq 0$.

\end{Definition}

It follows from the celebrated Birkhoff--Khinchin ergodic theorem \cite[Chapter 4, Section 6, Theorem 11]{Krylovstoch}
that any stationary sequence admits the limit 
\[
\frac{1}{N} \sum_{n=1}^N b(\vc{U}_n) (y) \to \Ov{b(\vc{U})}(y) \ \mbox{for a.a.} \ y \in Q 
\]
for any bounded Borel function $b : R^{D} \to R$. In particular, we get \eqref{S4}; whence any stationary sequence and all its statistical perturbations in the sense of Definition \ref{DQ1} are weakly (S)--convergent.

\subsection{Weak asymptotic stationarity} 

In practice, we do not expect the approximate sequences to be stationary, however, some kind of \emph{asymptotic} stationarity can be anticipated. 

\begin{Definition}[Weak asymptotic stationarity] \label{AD2}
A sequence $\{ \vc{U}_n \}_{n=1}^\infty,\ \vc{U}_n : Q \mapsto R^D$, is called \emph{weakly asymptotically stationary} if the following holds
for any $b \in C_c(R^D)$: 
\begin{itemize}
\item {\bf Correlation limit} 
\begin{equation} \label{A2}
\lim_{N \to \infty} \frac{1}{N} \sum_{n=1}^N \int_Q b(\vc{U}_n) b(\vc{U}_m) \ \D y \ \mbox{exists}
\end{equation}
for any fixed $m$; 
\item {\bf Asymptotic correlation stationarity}
\begin{equation} \label{A3}
\left| 
\int_Q \Big[ b(\vc{U}_{k_1}) b(\vc{U}_{k_2} ) - b(\vc{U}_{k_1 + n}) b(\vc{U}_{k_2+n} ) \Big] \D y \right| \leq \omega (b,k)
\end{equation}
for any $1 \leq k \leq k_1 \leq k_2$, and any $n \geq 0$, where 
\[
\omega (b,k) \to 0 \ \mbox{as}\ k \to \infty.
\]

\end{itemize}
\end{Definition}

Obviously, any stationary sequence is weakly asymptotically stationary. Definition \ref{AD2} is reminiscent of \emph{weak stationarity} 
in the stochastic sense, where expectations, second moments, and correlations are required to be stationary. Note however, 
that we require \eqref{A3} to hold for \emph{any} $b \in C_c(R^D)$. 

The next result shows that weakly asymptotically stationary sequences are weakly (S)--convergent.

\begin{Theorem} \label{AP1}

Let $\{ \vc{U}_n \}_{n=1}^\infty$ be a weakly asymptotically stationary sequence in the sense of Definition \ref{AD2}. 

Then $\{ \vc{U}_n \}_{n=1}^\infty$ is weakly (S)--convergent, in particular, 
\[
\frac{1}{N} \sum_{n=1}^N b(\vc{U}_n) \to \Ov{b (\vc{U})} \ \mbox{as}\ N \to \infty \ \mbox{in}\ L^1(Q) 
\]
for any $b \in C_c(R^D)$.

\end{Theorem}

\begin{proof}

It is enough to show that $\{ \vc{U}_n \}_{n=1}^\infty$ enjoys the property \eqref{S2}. We know from Lemma \ref{SL1} that 
\begin{equation} \label{A7}
\frac{1}{N} \sum_{n=1}^N b(\vc{U}_n) \to \Ov{b(\vc{U})} \ \mbox{as}\ N \to \infty \ \mbox{weakly-(*) in}\ L^\infty(Q); 
\end{equation}
whence \eqref{S2} reduces to showing 
\begin{equation} \label{A4}
\lim_{N \to \infty} \sum_{n,m = 1}^N \frac{1}{N^2} \int_Q b(\vc{U}_n) b(\vc{U}_m) \ \dy = 
\int_Q |\Ov{b(\vc{U})} |^2 \ \D y.
\end{equation}

Fixing $k > 0$ we first observe that 
\begin{equation} \label{A5}
\lim_{N \to \infty} \sum_{n,m = 1}^N \frac{1}{N^2} \int_Q b(\vc{U}_n) b(\vc{U}_m) \ \dy = 
\lim_{N \to \infty} \sum_{n,m = k}^{N + k} \frac{1}{N^2} \int_Q b(\vc{U}_n) b(\vc{U}_m) \ \dy. 
\end{equation}
Next, regrouping terms, we get 
\begin{equation} \label{A6}
\begin{split}
\sum_{n,m = k}^{N + k} b(\vc{U}_n) b(\vc{U}_m) &= 
\sum_{n = 0}^N \left( \sum_{m = 0}^N b(\vc{U}_{k + n}) b(\vc{U}_{k+m}) \right) \\ &= 
\sum_{n= 0}^N b(\vc{U}_k) \sum_{m = 0}^N b(\vc{U}_{k+m}) + 
\sum_{n = 0}^N  \sum_{m = 0}^N \Big( b(\vc{U}_{k + n}) b(\vc{U}_{k+m}) - b(\vc{U}_k) b(\vc{U}_{k+m}) \Big).
\end{split}
\end{equation}
Furthermore, 
\[
\begin{split}
\sum_{m=0}^N b(\vc{U}_{k + n}) b(\vc{U}_{k+m}) &= \sum_{m \leq n} b(\vc{U}_{k + n}) b(\vc{U}_{k+m}) + 
\sum_{m = n+1}^N b(\vc{U}_{k + n}) b(\vc{U}_{k+m}) \\ 
&= 
\sum_{m \leq n} \Big[  b(\vc{U}_{k + n}) b(\vc{U}_{k+m}) - b(\vc{U}_k) b(\vc{U}_{k + n-m}) \Big] + 
\sum_{m=0}^n b(\vc{U}_k) b(\vc{U}_{k + m})  \\ &+ 
\sum_{m = n+1}^N \Big[  b(\vc{U}_{k + n}) b(\vc{U}_{k+m}) - b(\vc{U}_k) b(\vc{U}_{k + m - n}) \Big]
+ \sum_{m= n+1}^N b(\vc{U}_k) b(\vc{U}_{k + m})\\ &=
\sum_{m \leq n} \Big[  b(\vc{U}_{k + n}) b(\vc{U}_{k+m}) - b(\vc{U}_k) b(\vc{U}_{k + n-m}) \Big] \\ 
&+ \sum_{m = n+1}^N \Big[  b(\vc{U}_{k + n}) b(\vc{U}_{k+m}) - b(\vc{U}_k) b(\vc{U}_{k + m - n}) \Big] 
+ \sum_{m=0}^N b(\vc{U}_k) b(\vc{U}_{k + m}). 
\end{split}
\]
Going back to \eqref{A6} we obtain 
\begin{equation} \label{A8}
\begin{split}
\sum_{n,m = k}^{N + k} b(\vc{U}_n) b(\vc{U}_m) &= 
(N + 1) b(\vc{U}_k) \sum_{m = 0}^N b(\vc{U}_{k+m})\\& + 
\sum_{n = 0}^N \sum_{m \leq n} \Big[  b(\vc{U}_{k + n}) b(\vc{U}_{k+m}) - b(\vc{U}_k) b(\vc{U}_{k + n-m}) \Big] \\
&+ \sum_{n = 0}^N \sum_{m = n+1}^N \Big[  b(\vc{U}_{k + n}) b(\vc{U}_{k+m}) - b(\vc{U}_k) b(\vc{U}_{k + m - n}) \Big]
\end{split}
\end{equation}

As $\{ \vc{U}_n \}_{n=1}^\infty$ is asymptotically stationary, it follows from \eqref{A3} that 
\[
\begin{split}
\left|
\int_Q \Big[  b(\vc{U}_{k + n}) b(\vc{U}_{k+m}) - b(\vc{U}_k) b(\vc{U}_{k + n-m}) \Big] \D y \right| \leq \omega(k),\ 
n \geq m \geq 0 \\ \left|
\int_Q
\Big[  b(\vc{U}_{k + n}) b(\vc{U}_{k+m}) - b(\vc{U}_k) b(\vc{U}_{k + m - n}) \Big] \D y \right| \leq \omega(k)\ 
m \geq n \geq 0
\end{split}
\]

Thus performing the limit in \eqref{A5} we obtain 
\begin{equation} \label{A9}
\begin{split}
\int_Q b(\vc{U}_k) \Ov{b(\vc{U})} \D y - \omega(k) &\leq 
\liminf_{N \to \infty} \sum_{n,m = 1}^N \frac{1}{N^2} \int_Q b(\vc{U}_n) b(\vc{U}_m) \ \dy\\ &\leq
\limsup_{N \to \infty} \sum_{n,m = 1}^N \frac{1}{N^2} \int_Q b(\vc{U}_n) b(\vc{U}_m) \ \dy \leq
\int_Q b(\vc{U}_k) \Ov{b(\vc{U})} \D y + \omega(k)
\end{split}
\end{equation}
where we have used the weak convergence stated in \eqref{A8}.

Finally, summing \eqref{A9} with respect to $k$, we get  
\[
\begin{split}
\frac{1}{M} \sum_{k=1}^M
&\int_Q b(\vc{U}_k) \Ov{b(\vc{U})} \D y - \frac{1}{M} \sum_{k=1}^M \omega(k) \leq 
\liminf_{N \to \infty} \sum_{n,m = 1}^N \frac{1}{N^2} \int_Q b(\vc{U}_n) b(\vc{U}_m) \ \dy\\ &\leq
\limsup_{N \to \infty} \sum_{n,m = 1}^N \frac{1}{N^2} \int_Q b(\vc{U}_n) b(\vc{U}_m) \ \dy \leq
\frac{1}{M} \sum_{k=1}^M\int_Q b(\vc{U}_k) \Ov{b(\vc{U})} \D y + \frac{1}{M} \sum_{k=1}^M\omega(k)
\end{split}
\]
Thus letting $M \to \infty$ and using \eqref{A7} once more, we get \eqref{S2}.

\end{proof}

\begin{Remark} \label{AR1}

As revealed in the proof of Proposition \ref{AP1}, the hypothesis \eqref{A3} can be replaced by a weaker stipulation:

\begin{equation} \label{A12} 
\limsup_{N \to \infty} \frac{1}{N^2} \sum_{n,m = 0}^{N} \left| \int_Q b(\vc{U}_{k + n}) b(\vc{U}_{k+m}) - b(\vc{U}_k) 
b (\vc{U}_{ k + |n-m| }) \D y \right| \leq \omega(b,k),
\end{equation} 
$\omega(b,k) \to 0$ as $k \to \infty$.

\end{Remark}

\subsection{Strong asymptotic stationarity} 

Our final goal in this section is to show a sufficient condition for strong (S)--convergence. 

\begin{Definition}[Strong asymptotic stationarity] \label{SuD2}
A sequence $\{ \vc{U}_n \}_{n=1}^\infty,\ \vc{U}_n : Q \mapsto R^D$, is called \emph{strongly asymptotically stationary} if the following holds
for any $b \in C_c(R^D)$: 
\begin{itemize}
\item {\bf Strong correlation limit} 
\begin{equation} \label{SuA2}
\lim_{n \to \infty} \int_Q b(\vc{U}_n) b(\vc{U}_m) \ \D y \ \mbox{exists}
\end{equation}
for any fixed $m$; 
\item {\bf Asymptotic correlation stationarity}
\begin{equation} \label{SuA3}
\left| 
\int_Q \Big[ b(\vc{U}_{k_1}) b(\vc{U}_{k_2} ) - b(\vc{U}_{k_1 + n}) b(\vc{U}_{k_2+n} ) \Big] \D y \right| \leq \omega (b,k)
\end{equation}
for any $1 \leq k \leq k_1 \leq k_2$, and any $n \geq 0$, where 
\[
\omega (b,k) \to 0 \ \mbox{as}\ k \to \infty.
\]

\end{itemize}
\end{Definition}

We claim the following analogue of Theorem \ref{AP1}

\begin{Theorem} \label{SuP1}

Let $\{ \vc{U}_n \}_{n=1}^\infty$ be strongly asymptotically stationary in the sense of Definition \ref{SuD2}. 

Then $\{ \vc{U}_n \}_{n=1}^\infty$ is strongly (S)--convergent, in particular, 
\[
\frac{1}{w_N} \sum_{n=1}^N w \left( \frac{n}{N} \right)  b(\vc{U}_n) \to \Ov{b (\vc{U})} \ \mbox{in}\ L^1(Q) 
\]
for any $b \in C_c(R^D)$, $w \in {W}$.

\end{Theorem}

\begin{proof}

In view of Lemma \ref{SL2}, it is enough to show that $\{ \vc{U}_n \}_{n=1}^\infty$ satisfies \eqref{Su4}. 
To begin, following the arguments of the proof of Lemma \ref{SL1} we observe that \eqref{SuA2} implies 
\begin{equation} \label{Su10}
b(\vc{U}_n) \to \Ov{b (\vc{U})} \ \mbox{weakly-(*) in}\ L^\infty(Q), 
\end{equation}
which yields 
\begin{equation} \label{Su7}
\frac{1}{w_N} \sum_{n=1}^N w \left( \frac{n}{N} \right) b(\vc{U}_n) \to \Ov{b(\vc{U})} \ \mbox{as}\ N \to \infty \ \mbox{weakly-(*) in}\ L^\infty(Q). 
\end{equation}
Consequently, it remains 
to show 
\begin{equation} \label{SuA4}
\lim_{N \to \infty} \sum_{n,m = 1}^N \frac{1}{w_N^2} w \left( \frac{n}{N} \right) w \left( \frac{m}{N} \right) 
\int_Q b(\vc{U}_n) b(\vc{U}_m) \ \dy = 
\int_Q |\Ov{b(\vc{U})} |^2 \ \D y.
\end{equation}

First, 
\begin{equation} \label{Su5A}
\begin{split}
\sum_{n,m = 1}^N \frac{1}{w_N^2} w &\left( \frac{n}{N} \right) w \left( \frac{m}{N} \right)  
\int_Q b(\vc{U}_n) b(\vc{U}_m) \ \dy - \int_Q |\Ov{b(\vc{U})} |^2 \ \D y\\ &= 
\sum_{n,m = 1}^N \frac{1}{w_N^2} w \left( \frac{n}{N} \right) w \left( \frac{m}{N} \right)  
\int_Q \Big[ b(\vc{U}_n) b(\vc{U}_m) -  |\Ov{b(\vc{U})} |^2 \Big] \D y \\&=
2 \sum_{n = 1}^{k-1} \sum_{m=1}^N \frac{1}{w_N^2} w \left( \frac{n}{N} \right) w \left( \frac{m}{N} \right)  
\int_Q \Big[ b(\vc{U}_n) b(\vc{U}_m) -  |\Ov{b(\vc{U})} |^2 \Big] \D y \\&+ 
\sum_{n,m = k}^{N}\frac{1}{w_N^2} w \left( \frac{n}{N} \right) w \left( \frac{m}{N} \right)  
\int_Q \Big[ b(\vc{U}_n) b(\vc{U}_m) -  |\Ov{b(\vc{U})} |^2 \Big] \D y
\end{split}
\end{equation}
Now observe that 
\begin{equation} \label{Su6A}
\frac{w_N}{N} = \frac{1}{N} \sum_{n=1}^N w \left( \frac{n}{N} \right) \to \int_0^1 w(z)\ \D z = 1 \ \mbox{as}\ N \to \infty,
\end{equation}
and, consequently, 
\[
\sum_{n = 1}^{k-1} \sum_{m=1}^N \frac{1}{w_N^2} w \left( \frac{n}{N} \right) w \left( \frac{m}{N} \right)  
\int_Q \Big[ b(\vc{U}_n) b(\vc{U}_m) -  |\Ov{b(\vc{U})} |^2 \Big] \D y \to 0 \ \mbox{as}\ N \to \infty 
\]
for any fixed $k$.

Going back to \eqref{Su5A}, we write the last integral as 
\[
\begin{split}
\sum_{n,m = k}^{N}\frac{1}{w_N^2} &w \left( \frac{n}{N} \right) w \left( \frac{m}{N} \right)  
\int_Q \Big[ b(\vc{U}_n) b(\vc{U}_m) -  |\Ov{b(\vc{U})} |^2 \Big] \D y\\ &= 
\sum_{n,m = k, |n-m| < l}^{N}\frac{1}{w_N^2} w \left( \frac{n}{N} \right) w \left( \frac{m}{N} \right)  
\int_Q \Big[ b(\vc{U}_n) b(\vc{U}_m) -  |\Ov{b(\vc{U})} |^2 \Big] \D y\\&+
\sum_{n,m = k, |n-m| \geq l}^{N}\frac{1}{w_N^2} w \left( \frac{n}{N} \right) w \left( \frac{m}{N} \right)  
\int_Q \Big[ b(\vc{U}_n) b(\vc{U}_m) -  |\Ov{b(\vc{U})} |^2 \Big] \D y
\end{split}
\]
Using \eqref{Su6A} once more, we get 
\[
\left| \sum_{n,m = k, |n-m| < l}^{N}\frac{1}{w_N^2} w \left( \frac{n}{N} \right) w \left( \frac{m}{N} \right)  
\int_Q \Big[ b(\vc{U}_n) b(\vc{U}_m) -  |\Ov{b(\vc{U})} |^2 \Big] \D y \right| \aleq \frac{lN}{w_N^2} \to 0 
\ \mbox{as}\ N \to \infty
\]
for any fixed $l$.

Finally,
\[ 
\begin{split}
\sum_{n,m = k, |n-m| \geq l}^{N}&\frac{1}{w_N^2} w \left( \frac{n}{N} \right) w \left( \frac{m}{N} \right)  
\int_Q \Big[ b(\vc{U}_n) b(\vc{U}_m) -  |\Ov{b(\vc{U})} |^2 \Big] \D y\\ &= 
2 \sum_{n,m = k, m \geq n + l}^{N}\frac{1}{w_N^2} w \left( \frac{n}{N} \right) w \left( \frac{m}{N} \right)  
\int_Q \Big[ b(\vc{U}_n) b(\vc{U}_m) -  |\Ov{b(\vc{U})} |^2 \Big] \D y\\ &= 
2 \sum_{n,m = k, m - n\geq l}^{N}\frac{1}{w_N^2} w \left( \frac{n}{N} \right) w \left( \frac{m}{N} \right)  
\int_Q \Big[ b(\vc{U}_k) b(\vc{U}_{m - n + k}) -  |\Ov{b(\vc{U})} |^2 \Big] \D y\\ &+ 
2 \sum_{n,m = k, m \geq n + l}^{N}\frac{1}{w_N^2} w \left( \frac{n}{N} \right) w \left( \frac{m}{N} \right)  
\int_Q \Big[ b(\vc{U}_n) b(\vc{U}_m) - b(\vc{U}_k) b(\vc{U}_{m - n + k})  \Big] \D y\\
\end{split}
\]
In view of hypothesis \eqref{SuA3}, 
\[
\left| \sum_{n,m = k, m \geq n + l}^{N}\frac{1}{w_N^2} w \left( \frac{n}{N} \right) w \left( \frac{m}{N} \right)  
\int_Q \Big[ b(\vc{U}_n) b(\vc{U}_m) - b(\vc{U}_k) b(\vc{U}_{m - n + k})  \Big] \D y \right| \leq \ep(b,k),
\]
where $\ep (b,k) \to 0$ as $k \to \infty$, uniformly for any $l$.
Now, in view of \eqref{Su10}, given $k$, we can fix $l = l(\ep (b,k) ,k)$ such that 
\[
\left| \sum_{n,m = k, m - n\geq l}^{N}\frac{1}{w_N^2} w \left( \frac{n}{N} \right) w \left( \frac{m}{N} \right)  
\int_Q \Big[ b(\vc{U}_k) b(\vc{U}_{m - n + k}) -  |\Ov{b(\vc{U})} |^2 \Big] \D y \right| < \ep (b,k).
\] 

\end{proof}

We strongly believe that asymptotic stationarity is satisfied by consistent approximations of the Euler system. For relevant numerical evidence, see \cite{FeiLukMizSheWa}.

\section{Applications, convergence of consistent approximations for the compressible Euler system}
\label{E} 

Our ultimate goal is to apply the abstract theory to the
\emph{isentropic Euler system}: 
\begin{equation}\label{E1} 
\begin{split}
\partial_t \vr + \Div \vm &= 0, \ \vr(0, \cdot) = \vr_0, \\ 
\partial_t \vm + \Div \left( \frac{\vm \otimes \vm}{\vr} \right) + \Grad p(\vr) &= 0, \ \vm(0, \cdot) = \vm_0,
\end{split}
\end{equation}
with the isentropic EOS 
$p(\vr) = a \vr^\gamma$, $\gamma > 1$. 
Here $\vr = \vr(t,x)$ is the mass density and $\vm = \vm(t,x)$ the linear momentum of a compressible gas in the isentropic regime.
For the sake of simplicity, we consider the periodic boundary conditions $\Omega = \Td$.
We recall the associated energy inequality in the integrated form
\[
\intTd{ E(\vr, \vm)(t, \cdot) } \leq \intTd{ E(\vr_0, \vm_0) },
\]
where 
\[
E(\vr, \vm) = \left\{ \begin{array}{l} \frac{1}{2} \frac{|\vm|^2}{\vr} + P(\vr) \ \mbox{if}\ \vr > 0,\\  
0 \ \mbox{if}\ \vr = 0, \ \vm = 0, \\  \infty \ \mbox{otherwise}\end{array} \right.
\]

\subsection{Dissipative solutions}
\label{D}

We recall the definition of \emph{dissipative solution} of the Euler system, see \cite{BreFeiHof19}: 
\begin{itemize} 
\item {\bf Equation of continuity}
\[
\partial_t \vr + \Div \vm = 0 \ \mbox{in}\ \mathcal{D}'((0,T) \times \Td);
\]
\item {\bf Momentum equation} 
\[
\partial_t \vm + \Div \left( 1_{\vr > 0} \frac{\vm \otimes \vm}{\vr} \right) + \Grad p(\vr) + \Div \mathfrak{R} = 0,\ 
\mbox{in}\ \mathcal{D}'((0,T) \times \Td; R^d), 
\]
with the Reynolds stress $\mathfrak{R} \in L^\infty(0,T; \mathcal{M}^+(\Td; R^{d \times d}_{\rm sym}))$;
\item {\bf Energy balance}
\[
\intTd{ E(\vr, \vm)(t, \cdot) } + \min\left\{ \frac{1}{2}; \frac{1}{\gamma} \right\} \int_{\Td} \D {\rm tr}[ 
\mathfrak{R}](t, \cdot) \leq \intTd{E(\vr_0, \vm_0)} 
\]
for a.a. $t \in (0,T)$.

\end{itemize}

It follows that 
\[
\vr \in C_{{\rm weak}}([0,T]; L^\gamma(\Td)),\ 
\vm \in C_{{\rm weak}}([0,T]; L^{\frac{2 \gamma}{\gamma + 1}}(\Td; R^d));
\]
whence the initial data are well defined. 

If $\mathfrak{R} \equiv 0$, the above definition yields the standard (admissible) weak solution. 
Note that the Reynolds stress $\mathfrak{R}$ is a positively definite matrix valued measure that accommodates possible oscillations/concentration inherited from the approximation process. Although definitely more general, the dissipative solutions share many 
important properties with the weak solutions, among which the weak--strong uniqueness principle. The reader may consult 
\cite{BreFeiHof19} or \cite{Fei2020} for other interesting properties of dissipative solutions. 

\subsection{Consistent approximation} 
\label{CA}

The class of dissipative solutions is large enough to accommodate limits of various approximate schemes that are \emph{consistent} with 
the Euler system, cf. \cite{MarEd} for the vanishing viscosity limit, and \cite{FeiLMMiz} for a finite volume numerical scheme. 

\begin{Definition}[Consistent approximation] \label{CAD1}

We say that a sequence $\{ \vr_n, \vm_n \}_{n=1}^\infty$ is \emph{consistent approximation} of the 
isentropic Euler system if the following holds:

\begin{itemize}
 
\item {\bf Approximate equation of continuity}
\[
\int_0^T \intTd{ \Big[ \vr_n \partial_t \varphi + \vm_n \cdot \Grad \varphi \Big] } \dt  =  
- \intTd{ \vr_{0,n} \varphi (0, \cdot) } + e^1_n [\varphi] 
\]
for any $\varphi \in C^\infty_c([0, T) \times \Td)$;

\item {\bf Approximate momentum equation} 
\[
\begin{split}
\int_0^T &\intTd{ 
\left[ \vm_n \cdot \partial_t \bfphi + \left( 1_{\vr_n > 0} \frac{\vm_n \otimes \vm_n}{\vr_n} \right) : 
\Grad \bfphi + p(\vr_n) \Div \bfphi + \mathfrak{R}_n : \Div \bfphi \right] } \dt \\ &= - 
\intTd{ \vm_{0,n} \cdot \bfphi(0, \cdot) } + e^2_n [\bfphi]
\end{split}
\]
for any $\bfphi \in C^\infty_c([0, T) \times \Td; R^d)$;
\item {\bf Approximate energy balance}
\[
\intTd{ E(\vr_n, \vm_n)(t, \cdot) } + \min\left\{ \frac{1}{2}; \frac{1}{\gamma} \right\} \int_{\Td} \D {\rm tr}[ 
\mathfrak{R}_n](t, \cdot) \leq E_n 
\]

\item {\bf Consistency} 
\[
\begin{split}
e^1_n[\varphi] &\to 0 \ \mbox{as}\ n \to \infty \ \mbox{for any}\ \varphi \in C^\infty_c([0, T) \times \Td),\\
e^2_n[\bfphi] &\to 0 \ \mbox{as}\ n \to \infty \ \mbox{for any}\ \bfphi \in C^\infty_c([0, T) \times \Td; R^d),\\ 
\vr_{0,n} &\to \vr_0 \ \mbox{weakly in}\ L^1(Q),\\ \vm_{0,n} &\to \vm_0 \ \mbox{weakly in}\ L^1(Q; R^d),\\
\limsup_{n \to \infty} E_n &\leq \intTd{ E(\vr_0, \vm_0) }. 
\end{split}
\]

\end{itemize}

\end{Definition}

Note that the above definition differs from \cite{MarEd}, \cite{FeiLMMiz} as it allows for ``approximate Reynolds stress'' 
$\mathfrak{R}_n$ that is set to be zero in \cite{MarEd}, \cite{FeiLMMiz}. Accordingly, the present definition accommodates a larger 
class of consistent approximations than \cite{MarEd}, \cite{FeiLMMiz}.
There are two crucial observations: 

\begin{enumerate}

\item Any weak limit of a sequence of consistent approximations $\{ \vr_n, \vm_n \}_{n = 1}^\infty$ is a dissipative solution with the initial data 
$(\vr_0, \vm_0)$. This follows from the general compactness results proved in \cite[Section 3, Proposition 3.1]{BreFeiHof19}. 

\item If $\{ \vr_n, \vm_n \}_{n =1}^\infty$ is a consistent approximation, then 
\[
\vr_N = \frac{1}{N} \sum_{n=1}^N \vr_n,\ \vm_N = \frac{1}{N}\sum_{n=1}^N \vm_n ,\ N = 1,2,\dots 
\]
is another consistent approximation (of the same problem, with the same data). This reflects a general principle that a convex combination of dissipative solutions is a dissipative solution. It remains to observe that  the error terms satisfy 
\[
\frac{1}{N} \sum_{n=1}^N e^1_n [\varphi],\ 
\frac{1}{N} \sum_{n=1}^N e^2_n [\bfphi] \to 0 \ \mbox{as}\ N \to \infty.
\]

\end{enumerate}

\begin{Corollary}[Convergence of (S)--convergent consistent approximations]  \label{DC1}

Suppose that a sequence $\{ \vr_n, \vm_n \}_{n =1}^\infty$ is a consistent approximations of the Euler system that is 
weakly (S)--convergent  
in the sense of Definition \ref{SD1}. 

Then 
\begin{itemize}
\item 
\[
\begin{split}
\frac{1}{N} \sum_{n=1}^\infty \vr_n &\to \vr \ \mbox{in}\ L^q((0,T) \times \Td),\ 1 \leq q < \gamma, \\ 
\frac{1}{N} \sum_{n=1}^\infty \vm_n &\to \vm \ \mbox{in}\ L^q((0,T) \times \Td),\ 1 \leq q < \frac{2 \gamma}{\gamma + 1},  
\end{split}
\]
where $(\vr, \vm)$ is a dissipative solution of the Euler system;  
\item there exists a (unique) parametrized measure $( \mathcal{V}_{t,x} )_{(t,x) \in (0,T) \times \Td)}$, 
$\mathcal{V}_{t,x} \in \mathcal{P}(R^{d+1})$,
such that 
\[
(\vr_n, \vm_n) \toS \mathcal{V},
\]
and
\[
\int_0^T \intTd{
d_{W_s} \left[ 
\frac{1}{N} \sum_{n=1}^N \delta_{\vr(t,x), \vm(t,x)};  \mathcal{V}_{t,x} \right]^s } \dt \to 0 
\ \mbox{for any}\ 1 \leq s < \frac{2 \gamma}{\gamma + 1};  
\] 

\item If, in addition, the sequence $\{ \vr_n, \vm_n \}_{n =1}^\infty$ is strongly (S)--convergent, the 
\[
\frac{1}{N} \sum_{n=1}^N \ \mbox{may be replaced by}\ \frac{1}{w_N} \sum_{n=1}^N w \left( \frac{n}{N} \right) 
\]
for arbitrary weight $w \in W$.

\end{itemize}

\end{Corollary}

The parametrized measure $\mathcal{V}$ can be seen as a generalized solution of the Euler system. It shares all 
fundamental properties with the Young measures 
generated by (sub)sequences of consistent approximations. In particular, 
\[
\mathcal{V}_{(t,x)} = \delta_{[\vr, \vm](t,x)} \ \mbox{for a.a.}\ (t,x) \in (0,T) \times \Td 
\]
if either the limit Euler system admits a smooth solution, or the barycenter of $\mathcal{V}$ -- $[\vr, \vm]$ -- is 
of class $C^1$.

\subsection{(S)--convergence to weak solutions}
\label{W}

We know that weakly (S)--convergent approximate sequences to the isentropic Euler system generate the parametrized measure 
$\mathcal{V}$. Our ultimate goal is to discuss validity of the following statement:
\begin{equation} \label{W1}
\begin{split}
\vr = \left< \mathcal{V}; \tvr \right>,\ 
\vm = \left< \mathcal{V}; \tvm \right> \ &\mbox{is a weak solution to the Euler system} \\
& \Rightarrow \mathcal{V}_{t,x} = \delta_{(\vr, \vm)(t,x)} \ \mbox{for a.a.}\ (t,x) \in (0,T) \times \Td.
\end{split}
\end{equation}

First observe that if $(\vr, \vm)$ is a weak (distributional) solution of the Euler system, then the Reynolds defect tensor 
satisfies 
\[
\Div \mathfrak{R} = 0 \ \mbox{in}\ \mathcal{D}'((0,T) \times \Td; R^d)
\]
from which we easily deduce 
\begin{equation} \label{W2}
\Div \mathfrak{R} (t, \cdot) = 0 \ \mbox{in}\ \mathcal{D}'(\Td),\ \mathfrak{R} \in \mathcal{M}^+(\Td; R^{d \times d}_{\rm sym})
\ \mbox{for a.a.}\ t \in (0,T).
\end{equation}
Moreover, the limit Reynolds stress can be written as 
\[
\mathfrak{R} = \mathfrak{R}_1 + \mathfrak{R}_2, 
\]
where 
\[
\mathfrak{R}_1 = \lim_{k \to \infty} \frac{1}{N_k} \sum_{n=1}^{N_k} \mathfrak{R}_{n} \ \mbox{--weak-(*) in}\ 
L^\infty(0,T; \mathfrak{M}(\Td; R^{d \times d}_{\rm sym})),\ \mathfrak{R}_1(t, \cdot) \in \mathcal{M}^+(\Td; R^{d \times d}_{\rm sym}),
\] 
\[
\begin{split}
\mathfrak{R}_2 = \lim_{k \to \infty} \frac{1}{N_k} &\sum_{n=1}^{N_k}\left[  \left( 1_{\vr_n > 0} \left( \frac{\vm_n \otimes \vm_n}{\vr_n} 
\right) + p(\vr_n) \mathbb{I} \right) - \left( 1_{\vr > 0} \left( \frac{\vm \otimes \vm}{\vr} 
\right) + p(\vr) \mathbb{I} \right)  \right] \\ &\mbox{--weak-(*) in}\ 
L^\infty(0,T; \mathfrak{M}(\Td; R^{d \times d}_{\rm sym})),
\end{split}
\]
where 
\[
\mathfrak{R}_2(t, \cdot) \in \mathcal{M}^+(\Td; R^{d \times d}_{\rm sym})\ \mbox{for a.a}\ t \in (0,T).
\]

Thanks to the convexity argument, specified in \cite{MarEd}, implication \eqref{W1} follows as soon as we can show that $\mathfrak{R} = 0$
Moreover, as $\mathfrak{R}$ is positively semi--definite and satisfies \eqref{W2}, it is enough to show that 
$\mathfrak{R}$ vanishes in a neighborhood of the boundary of $\Omega = \Td$, see \cite[Section 4, Proposition 4.3]{MarEd}.
As $\mathfrak{R}_1$ usually vanishes for consistent 
approximation of the Euler system, we have to make sure that $\mathfrak{R}_2$ vanishes in a neighborhood of 
$\partial \Omega$. Note that here we identify $\Td$ with a bounded subset of $R^d$. Following step by step the arguments of 
\cite[Section 4]{MarEd}, we can show the following result.  

\begin{Theorem} \label{WT1}

Let $\{ \vr_n, \vm_n \}_{n=1}^\infty$ be a consistent approximation of the isentropic Euler system
in the sense of Definition \ref{CAD1}, with $\mathfrak{R}_n \equiv 0$. Moreover, suppose that
\[
(\vr_n, \vm_n) \toS \mathcal{V}, 
\]
and 
\begin{equation} \label{W3}
\frac{1}{N} \sum_{n=1}^N \int_0^T \int_{\mathcal{U}}  E(\vr_n, \vm_n) \ \dxdt \to \int_0^T \int_{\mathcal{U}} E(\vr, \vm) \ \dxdt, 
\ \vr = \left< \mathcal{V}; \tvr \right>, \ \vm = \left< \mathcal{V}; \tvm \right>,   
\end{equation}
where $\mathcal{U}$ is an open neighborhood of $\partial \Omega$, $\Omega = \Td$. 
Finally, suppose that $(\vr, \vm)$ is a weak solution of the Euler system. 

Then 
\[
\mathcal{V}_{(t,x)} = \delta_{(\vr, \vm)(t,x)}\ \mbox{for a.a.} \ (t,x) \in (0,T) \times \Td.
\]

\end{Theorem}

We finish this part by a short discussion when hypothesis \eqref{W3} can be anticipated.
Suppose that 
\[
\vr_0 = \Ov{\vr} > 0, \ \vm = \Ov{\vm} \ \mbox{for}\ |x| \geq R.
\]
Supposing the finite--speed of propagation for the Euler system, we may infer that 
\[
E(\vr, \vm) = E(\Ov{\vr}, \Ov{\vm}) \ \mbox{for all}\ t \in (0,T),\ |x| \leq R + ct. 
\]
Consistently, we assume that 
\begin{equation} \label{W4}
E(\vr_n, \vm_n) \to E(\Ov{\vr}, \Ov{\vm}) \ \mbox{in}\ L^1 \left\{ t \in (0,T),\ |x| \leq R + ct \right\}
\end{equation}
In numerical approximations, \eqref{W4} is usually guaranteed by imposing a (CFL) condition.

Finally, we claim that hypothesis \eqref{W3} can be dropped in the case of complete Euler system, provided the reference variables 
are the density $\vr$, the momentum $\vm$, and the total entropy $S$. The interested reader may elaborate the details following 
\cite[Section 2]{MarEd}. With a bit of extrapolation, we may conclude that if the barycenter of the (S)--limit $\mathcal{V}$ is a weak solution of the Euler system, then $\mathcal{V}$ is a (parametrized) Dirac mass. 

\section{Conclusion}

The (S)--convergence provides a tool to study the limits of approximate sequences even in the case when the weak limit does not exist. This is in particular convenient for numerical schemes, where the procedure of picking up a suitable subsequence is 
practically not applicable. The notion is stable under very general statistical perturbations that may ``polute'' the approximation procedure. The limit measure $\mathcal{V}$ is attained in the strong topology of the underlying physical space and in the space of probability measures endowed with suitable Wasserstein distance. In particular, deviation, variance, barycenter and other 
parameters of the limit measure can be effectively computed. 

Given a sequence of consistent approximations, it is a hard problem to determine whether or not it is (S)--convergent. Note that it is possible to construct examples of consistent approximations that are not (S)--convergent at least for certain class of initial data. 
Given a consistent approximation resulting from the vanishing viscosity process or as a limit of a specific numerical scheme, a rigorous verification of validity of any form (weak or strong) of (S)--convergence remains an outstanding open problem.

\def\cprime{$'$} \def\ocirc#1{\ifmmode\setbox0=\hbox{$#1$}\dimen0=\ht0
  \advance\dimen0 by1pt\rlap{\hbox to\wd0{\hss\raise\dimen0
  \hbox{\hskip.2em$\scriptscriptstyle\circ$}\hss}}#1\else {\accent"17 #1}\fi}

%\bibliography{citace}

\begin{thebibliography}{10}

\bibitem{Bald}
E.~J. Balder.
\newblock Lectures on {Y}oung measure theory and its applications in economics.
\newblock {\em Rend. Istit. Mat. Univ. Trieste}, {\bf 31}(suppl. 1):1--69,
  2000.
\newblock Workshop on Measure Theory and Real Analysis (Italian) (Grado, 1997).

\bibitem{BreFeiHof19}
D.~Breit, E.~Feireisl, and M.~Hofmanov\'{a}.
\newblock Solution semiflow to the isentropic {E}uler system.
\newblock {\em Arch. Ration. Mech. Anal.}, {\bf 235}(1):167--194, 2020.

\bibitem{BresMur}
A.~Bressan and R.~Murray.
\newblock On self--similar solutions to the incompressible {E}uler equations.
\newblock 2020.
\newblock Preprint.

\bibitem{BucVic}
T.~Buckmaster and V.~Vicol.
\newblock Convex integration and phenomenologies in turbulence.
\newblock {\em Arxive Preprint Series}, {\bf arXiv 1901.09023v2}, 2019.

\bibitem{Chiod}
E.~Chiodaroli.
\newblock A counterexample to well-posedness of entropy solutions to the
  compressible {E}uler system.
\newblock {\em J. Hyperbolic Differ. Equ.}, {\bf 11}(3):493--519, 2014.

\bibitem{DaYo}
S.~Das and J.~A. Yorke.
\newblock Super convergence of ergodic averages for quasiperiodic orbits.
\newblock {\em Nonlinearity}, {\bf 31}(2):491--501, 2018.

\bibitem{DelSze3}
C.~De~Lellis and L.~Sz{\'e}kelyhidi, Jr.
\newblock On admissibility criteria for weak solutions of the {E}uler
  equations.
\newblock {\em Arch. Ration. Mech. Anal.}, {\bf 195}(1):225--260, 2010.

\bibitem{DiPMaj87a}
R.~J. DiPerna and A.~J. Majda.
\newblock Oscillations and concentrations in weak solutions of the
  incompressible fluid equations.
\newblock {\em Comm. Math. Phys.}, {\bf 108}(4):667--689, 1987.

\bibitem{DiP2}
R.J. Di{P}erna.
\newblock Measure-valued solutions to conservation laws.
\newblock {\em Arch. Rat. Mech. Anal.}, {\bf 88}:223--270, 1985.

\bibitem{Fei2020}
E.~Feireisl.
\newblock A note on the long--time behavior of dissipative solutions to the
  {E}uler system.
\newblock {\em Preprint Series IM AS, Praha}, {\bf preprint No. IM-2020-4},
  2020.

\bibitem{MarEd}
E.~Feireisl and M.~Hofmanov{\'a}.
\newblock On convergence of approximate solutions to the compressible {E}uler
  system.
\newblock {\em Arxive Preprint Series}, {\bf arXiv 1905.02548}, 2019.

\bibitem{FeiLMMiz}
E.~Feireisl, M.~Luk{\' a}{\v c}ov{\' a}-{M}edvid'ov{\' a}, and H.~Mizerov{\'
  a}.
\newblock $\mathcal{K}-$convergence as a new tool in numerical analysis.
\newblock {\em Arxive Preprint Series}, {\bf arxiv preprint No. 1904.00297},
  2019.

\bibitem{FeiLukMizSheWa}
E.~Feireisl, M.~Luk{\' a}{\v c}ov{\' a}-{M}edvid'ov{\' a}, H.~Mizerov{\' a},
  B.~She, and Y.~Wang.
\newblock Computing oscillatory solutions to the {E}uler system via
  $\mathcal{K}$-convergence.
\newblock {\em Arxive Preprint Series}, {\bf arxiv preprint No. 1910.03161},
  2019.

\bibitem{FeiLuk}
E.~Feireisl and M.~Luk\'{a}\v{c}ov\'{a}-Medvid'ov\'{a}.
\newblock Convergence of a mixed finite element--finite volume scheme for the
  isentropic {N}avier-{S}tokes system via dissipative measure-valued solutions.
\newblock {\em Found. Comput. Math.}, {\bf 18}(3):703--730, 2018.

\bibitem{FjKaMiTa}
U.~K. Fjordholm, R.~K{\" a}ppeli, S.~Mishra, and E.~Tadmor.
\newblock Construction of approximate entropy measure valued solutions for
  hyperbolic systems of conservation laws.
\newblock {\em Foundations Comp. Math.}, pages 1--65, 2015.

\bibitem{FjMiTa1}
U.~S. Fjordholm, S.~Mishra, and E.~Tadmor.
\newblock On the computation of measure-valued solutions.
\newblock {\em Acta Numer.}, {\bf 25}:567--679, 2016.

\bibitem{GSWW}
P.~Gwiazda, A.~\'Swierczewska-Gwiazda, and E.~Wiedemann.
\newblock Weak-strong uniqueness for measure-valued solutions of some
  compressible fluid models.
\newblock {\em Nonlinearity}, {\bf 28}(11):3873--3890, 2015.

\bibitem{Kom}
J.~Koml\'{o}s.
\newblock A generalization of a problem of {S}teinhaus.
\newblock {\em Acta Math. Acad. Sci. Hungar.}, {\bf 18}:217--229, 1967.

\bibitem{Kren}
U.~Krengel.
\newblock On the speed of convergence in the ergodic theorem.
\newblock {\em Monatsh. Math.}, {\bf 86}(1):3--6, 1978/79.

\bibitem{Krylovstoch}
N.~V. Krylov.
\newblock {\em Introduction to the theory of random processes}, volume~43 of
  {\em Graduate Studies in Mathematics}.
\newblock American Mathematical Society, Providence, RI, 2002.

\bibitem{SaLiFe}
F.~Le\'{o}n-Saavedra, M.~del~Carmen List\'{a}n-Garc\'{\i}a, F.~J.
  P\'{e}rez~Fern\'{a}ndez, and M.~P. Romero de~la Rosa.
\newblock On statistical convergence and strong {C}es\`aro convergence by
  moduli.
\newblock {\em J. Inequal. Appl.}, pages Paper No. 298, 12, {\bf 2019}.

\bibitem{Wied}
E.~Wiedemann.
\newblock Existence of weak solutions for the incompressible {E}uler equations.
\newblock {\em Ann. Inst. H. Poincar\'e Anal. Non Lin\'eaire}, {\bf
  28}(5):727--730, 2011.

\end{thebibliography}
%\bibliographystyle{plain}

\end{document}